\documentclass[a4paper,10pt,article]{amsart}

\usepackage{amsmath, amscd,amssymb,amsthm,CJKnumb,verbatim,indentfirst,dsfont,mathrsfs,ipa,graphicx,textcomp,enumerate}
\usepackage[all]{xy}
\usepackage{CJK,ipa,amsmath,booktabs,longtable}
\usepackage[numbers]{natbib}
\usepackage[colorlinks=true,
            linkcolor=blue,
            anchorcolor=blue,
            citecolor=blue]{hyperref}
\newtheorem{theorem}{Theorem}
\newtheorem{definition}[theorem]{Definition}
\newtheorem{remark}[theorem]{Remark}
\newtheorem{lemma}[theorem]{Lemma}
\newtheorem{corollary}[theorem]{Corollary}
\newtheorem{proposition}[theorem]{Proposition}

\newtheorem{example}[theorem]{Example}
\newtheorem{question}[theorem]{Question}

\title{$B^p_r(F_n)$ has no nontrivial idempotents}

\author{Yifan Liu}
\address{Research Center for Operator Algebras, School of Mathematical Sciences, East China Normal University, Shanghai 200062, China}
\email{fenix6b1s3@163.com}

\author{Jianguo Zhang}
\address{Research Center for Operator Algebras, School of Mathematical Sciences, East China Normal University, Shanghai 200062, China}
\email{jgzhang@math.ecnu.edu.cn}

\date{}

\begin{document}

\maketitle

\begin{abstract}
    We show that there is no nontrivial idempotent in the reduced group $\ell^p$-operator algebra $B^p_r(F_n)$ of the free group $F_n$ on $n$ generators for each positive integer $n$.
\end{abstract}

\pagestyle{plain}
\section{Introduction}

    For a unital Banach algebra $A$, an \textit{idempotent} in $A$ is an element $a$ with $a^2=a$. Obviously, both the zero element $0$ and the unit element $I$ are idempotents in $A$. Besides, an idempotent $a$ in $A$ is called to be $\textit{nontrivial idempotent}$, if $a$ is neither $0$ nor $I$. Around 1949, Kadison and Kaplansky conjectured that if $\Gamma$ is a torsion free discrete group, then the reduced group $C^{\ast}$-algebra $C^{\ast}_r(\Gamma)$ has no nontrivial idempotents (or projections, equivalently). Through the decades, people have made great achievements around the conjecture, but whether it is true is still unknown. 

    An important way to investigate the Kadison-Kaplansky conjecture is through the Baum-Connes conjecture, namely, if $\Gamma$ is a torsion free discrete group that satisfies the Baum-Connes conjecture (actually, surjectivity of the assembly map is sufficient), then $C^{\ast}_r(\Gamma)$ has no nontrivial idempotents (cf. \cite{BCH1}). This includes a large class of groups. For example, Higson and Kasparov (cf. \cite{HK1}) showed that the Baum-Connes conjecture is true for a-T-menable groups, which include amenable groups and free groups, Lafforgue (cf. \cite{Lafforgue2002}) and Mineyev, Yu (cf. \cite{MineyevYu}) proved that the Baum-Connes conjecture holds for hyperbolic groups. Hence if $\Gamma$ is torsion free, discrete a-T-menable groups or hyperbolic groups, then $C^{\ast}_r(\Gamma)$ contains no nontrivial idempotents. For hyperbolic groups, we note that there is another way to study the Kadison-Kaplansky conjecture, which is due to Puschnigg by using of local cyclic homology (cf. \cite{Puschnigg2002}).

    As people are more interested in free groups, the stories of the case when $\Gamma$ is a free group $F_n$ could be even longer. It was first shown by Pimsner and Voiculescu (cf. \cite{PV1}) that $C^{\ast}_r(F_n)$ is without nontrivial idempotents, and then by others (cf. \cite{Connes1}, \cite{CF1}). 
    
    For a discrete group $\Gamma$, if we consider the left regular representation of $\Gamma$ on $\ell^p$-space $\ell^p(\Gamma)$, then we can define the reduced group $\ell^p$-operator algebra $B^p_r(\Gamma)$ for $p\in [1,\infty]$ (cf. Definition \ref{deflpalg}). And if $\Gamma$ is the free group $F_n$, Phillips raised a question concerning on the existence of nontrivial idempotents in $B^p_r(F_n)$ (cf. \cite{Phillips1}):

\begin{question}(\cite[part of problem 9.3]{Phillips1})\label{Phillipsquestion}
    For $n\in\{2,3,\dots\}$ and $p\in[1,\infty)$, does $B^p_r(F_n)$ have nontrivial idempotents?
\end{question}

    In this paper we give an answer to the above question (cf. Example \ref{answertoPhillips}) by using of property $(RD)_q$ of groups (cf. Definition \ref{defRD}) introduced by Liao and Yu (cf. \cite{LY1}). The main results of this paper are as follows:
\begin{theorem}(cf. Theorem \ref{maintheorem})
    Let $p\in [1,\infty]$, $q$ be its dual number and $\Gamma$ be a discrete group. Assume $\Gamma$ has property $(RD)_q$, if $C^{\ast}_r(\Gamma)$ has no nontrivial idempotents, then both $B^p_r(\Gamma)$ and $B^q_r(\Gamma)$ also have no nontrivial idempotents.
\end{theorem}
    Since that groups with property $(RD)$ have property $(RD)_q$ for any $q\in [1,2]$, thus we have the following corollary:
\begin{corollary}(cf. Corollary \ref{RDidempotent})
    Let $\Gamma$ be a discrete group. Assume $\Gamma$ has property $(RD)$, if $C^{\ast}_r(\Gamma)$ has no nontrivial idempotents, then for any $p\in [1,\infty]$, $B^p_r(\Gamma)$ also has no nontrivial idempotents. 
\end{corollary}
    Apart from free groups, the above results also apply to torsion free hyperbolic groups (cf. Example \ref{hyperbolicgroupcase}), torsion free groups with polynomial growth and torsion free cocompact lattices of $SL(3,\mathbb{R})$ (cf. Example \ref{polySL}). 

\section{Preliminaries}
    In this section, we will recall some notions. Let $p\in[1,\infty]$ and $\Gamma$ be a discrete group, recall that the \textit{group algebra} $\mathbb{C}\Gamma$ is the algebra of all finitely supported functions $f:\Gamma \rightarrow \mathbb{C}$ equipped with the following multiplication:
    $$f\ast g :=\sum_{\alpha,\gamma\in\Gamma}(f_{\alpha}g_{\gamma})(\alpha\gamma),$$
for any two elements $f=\sum_{\alpha\in \Gamma}f_{\alpha}\alpha$ and $g=\sum_{\gamma \in \Gamma}g_{\gamma}\gamma$ in $\mathbb{C}\Gamma$. The \textit{left regular representation} of $\Gamma$ on $\ell^p(\Gamma)$, denoted by $\lambda:\Gamma \rightarrow \mathcal{B}(\ell^p(\Gamma))$, is defined to be 
    $$(\lambda(\gamma)\xi)(\alpha):=\xi(\gamma^{-1}\alpha),$$   
for any $\gamma,\alpha \in \Gamma$ and $\xi \in \ell^p(\Gamma)$. For any $f=\sum_{\alpha\in \Gamma}f_{\alpha}\alpha$ in $\mathbb{C}\Gamma$, the \textit{reduced norm} of $f$, denoted by $\|f\|_{\mathcal{B}(\ell^p(\Gamma))}$ is defined to be 
    $$\|f\|_{\mathcal{B}(\ell^p(\Gamma))}:=\|\sum_{\alpha \in \Gamma}f_{\alpha}\lambda(\alpha)\|.$$
\begin{definition}\label{deflpalg}
    Let $p\in [1,\infty]$ and $\Gamma$ be a discrete group. 
  \begin{enumerate}[(1)]
    \item The \textit{reduced group $\ell^p$-operator algebra} of $\Gamma$, denoted by $B^p_r(\Gamma)$, is defined to be the reduced norm closure of the group algebra $\mathbb{C}\Gamma$.
    \item For any $f=\sum_{\alpha\in \Gamma}f_{\alpha}\alpha$ in $\mathbb{C}\Gamma$, let $f^{\ast}=\sum_{\alpha\in \Gamma}\overline{f_{\alpha}}\alpha^{-1}$, the \textit{reduced group involution $\ell^p$-operator algebra} of $\Gamma$, denoted by $B^{p,\ast}_r(\Gamma)$, is defined to be the completion of $\mathbb{C}\Gamma$ with respect to the following norm
    $$\|f\|_{B^{p,\ast}_r(\Gamma)}:=\max\{\|f\|_{\mathcal{B}(\ell^p(\Gamma))}, \|f^{\ast}\|_{\mathcal{B}(\ell^p(\Gamma))}\}.$$
  \end{enumerate}
\end{definition}

\begin{remark}\label{remark1}
    The reduced group $\ell^p$-operator algebra $B^p_r(\Gamma)$ is a Banach algebra, and generally, they are not the same for different $p$ (cf. \cite{LY1} and references therein). The reduced group involution $\ell^p$-operator algebra $B^{p,\ast}_r(\Gamma)$, defined by Liao and Yu in \cite{LY1}, is a Banach $\ast$-algebra and there exist some groups $G$ such that $B^{p,\ast}_r(G) \neq B^p_r(G)$ for $p\neq 2$ (cf. \cite{LY1} and references therein). 
\end{remark}

\begin{remark}\label{remark2}
    For $p\in (1,\infty)$, let $q$ be its dual number (i.e. $1/p+1/q=1$), then the dual space of $\ell^p(\Gamma)$ is $\ell^q(\Gamma)$, thus if $f^{\ast}$ is a bounded operator on $\ell^p(\Gamma)$, then $f$ is a bounded operator on $\ell^q(\Gamma)$ and we have $\|f^{\ast}\|_{\mathcal{B}(\ell^p(\Gamma))}=\|f\|_{\mathcal{B}(\ell^q(\Gamma))}$ for any $f\in \mathbb{C}\Gamma$, as a consequence, we have $B^{p,\ast}_r(\Gamma)=B^{q,\ast}_r(\Gamma)$. Obviously, we have $B^{1,\ast}_r(\Gamma)= B^1_r(\Gamma)=\ell^1(\Gamma)$ and $B^{\infty,\ast}_r(\Gamma)= B^{\infty}_r(\Gamma)=\ell^1(\Gamma)$. Besides, when $p=2$, then $B^{2,\ast}_r(\Gamma)= B^2_r(\Gamma)$ is a $C^{\ast}$-algebra, called the \textit{reduced group $C^{\ast}$-algebra} of $\Gamma$, which we shall denoted by $C^{\ast}_r(\Gamma)$. 
\end{remark}

    For any discrete group $\Gamma$, we have the following relation between $B^{p,\ast}_r(\Gamma)$ and $B^p_r(\Gamma)$.
\begin{lemma}\label{involutionsub}
    Let $p\in [1,\infty]$, then the identity map on $\mathbb{C}\Gamma$ extends to a contractive, injective homomorphism of Banach algebras:
      $$\iota_{\ast,p}: B^{p,\ast}_r(\Gamma)\rightarrow B^p_r(\Gamma),$$ 
    namely, $B^{p,\ast}_r(\Gamma)$ is contained in $B^p_r(\Gamma)$. Similarly, $B^{p,\ast}_r(\Gamma)$ is also contained in $B^q_r(\Gamma)$, where $q$ be the dual number of $p$.
\end{lemma}
\begin{proof}
    Obviously, $\iota_{\ast,p}$ is a contractive homomorphism, now we prove that it is injective. Assume $\iota_{\ast,p}(T)=0$ and a family of elements $\{f_i\}_{i\in I}$ in $\mathbb{C}\Gamma$ which converges to $T$ in $B^{p,\ast}_r(\Gamma)$. Since $\max\{\|f_i\|_{\ell^p(\Gamma)}, \|f_i\|_{\ell^q(\Gamma)}\}\leq \|f_i\|_{B^{p,\ast}_r(\Gamma)}$, thus $\{f_i\}_{i\in I}$ converges to an element $f$ in $\ell^p(\Gamma)\cap \ell^q(\Gamma)$ (where $q$ is the dual number of $p$) and we have $T=\lambda(f)$ in $B^{p,\ast}_r(\Gamma)$ which means 
  \begin{equation}\label{operator-vector}
    T=\sum_{\alpha\in \Gamma}f_{\alpha}\lambda(\alpha),
  \end{equation}
as operators on $\ell^p(\Gamma)$ and $\ell^q(\Gamma)$, where $f=\sum_{\alpha\in \Gamma}f_{\alpha}\alpha$ (which may be an infinite sum) and $\lambda$ be the left regular representation of $\Gamma$ on $\ell^p(\Gamma)$ and $\ell^q(\Gamma)$. By the assumption, $T=0$ as an operator on $\ell^p(\Gamma)$, thus $f=0$ as a vector in $\ell^p(\Gamma)\cap \ell^q(\Gamma)$ which implies $f_{\alpha}=0$ for any $\alpha\in \Gamma$. Then by (\ref{operator-vector}), we obtain $T=0$ in $B^{p,\ast}_r(\Gamma)$ which implies that $\iota_{\ast,p}$ is an injective homomorphism.\par
    By the similar discussion, we obtain that $B^{p,\ast}_r(\Gamma)$ is also contained in $B^q_r(\Gamma)$. 
\end{proof}

    The following proposition is due to Liao and Yu (cf. \cite[Proposition 2.4]{LY1}). For convenience of the reader, we show its proof.
    
\begin{proposition}(\cite[Proposition 2.4]{LY1})\label{interpolation}
     Let $p\in [1,\infty]$ and $\Gamma$ be a discrete group, then the identity map on $\mathbb{C}\Gamma$ extends to a contractive, injective homomorphism of Banach algebras
      $$\iota_{p,2}: B^{p,\ast}_r(\Gamma)\rightarrow C^{\ast}_r(\Gamma).$$
\end{proposition} 
\begin{proof}
    By Remark \ref{remark2}, any element $T\in B^{p,\ast}_r(\Gamma)$ is not only a bounded operator on $\ell^p(\Gamma)$ with the norm less than $\|T\|_{B^{p,\ast}_r(\Gamma)}$, but also a bounded operator on $\ell^q(\Gamma)$ with the norm less than $\|T\|_{B^{p,\ast}_r(\Gamma)}$, where $q$ is the dual number of $p$. Thus, by Riesz-Thorin interpolation theorem, $T$ is a bounded operator on $\ell^2(\Gamma)$ with the norm less than $\|T\|_{B^{p,\ast}_r(\Gamma)}$, which implies that $\iota_{p,2}$ is a contractive homomorphism.\par
    Now we show that $\iota_{p,2}$ is injective. Assume $\iota_{p,2}(T)=0$, as the proof of Lemma \ref{involutionsub}, there exists a vector $f\in \ell^p(\Gamma)\cap \ell^q(\Gamma)$ such that $T=\lambda(f)$ as operators on $\ell^p(\Gamma)$ and $\ell^q(\Gamma)$. By Riesz-Thorin interpolation theorem again, we have $T=\lambda(f)$ as operators on $\ell^2(\Gamma)$. By the assumption, $T=0$ in $C^{\ast}_r(\Gamma)$ which implies that $f=0$, thus $T=0$ in $B^{p,\ast}_r(\Gamma)$ which implies that $\iota_{p,2}$ is an injective homomorphism.
\end{proof}
    Recall that an idempotent in a unital Banach space is called nontrivial, if it is neither the zero element $0$ nor the unit element $I$. 
\begin{corollary}\label{Cimplyp}
    Let $p\in [1,\infty]$ and $\Gamma$ be a discrete group, if $C^{\ast}_r(\Gamma)$ has no nontrivial idempotents, then $B^{p,\ast}_r(\Gamma)$ also has no nontrivial idempotents.
\end{corollary}
\begin{proof}
    Assume that $e$ is a nontrivial idempotent in $B^{p,\ast}_r(\Gamma)$, then $\iota_{p,2}(e)$ is a nontrivial idempotent in $C^{\ast}_r(\Gamma)$, since that $\iota_{p,2}$ is an injective homomorphism, we get a contradiction.
\end{proof}

\section{Idempotents and Property $(RD)_q$}

    In this section, we will explore nontrivial idempotents in $B^p_r(\Gamma)$ and in $B^q_r(\Gamma)$ from nontrivial idempotents in $C^{\ast}_r(\Gamma)$ by using of property $(RD)_q$ of group $\Gamma$, where $q\in[1,2]$ and $p$ is the dual number of $q$. Firstly, we give the following key lemma.
    


\begin{lemma}\label{connectedspectrum}
    Let $A$ be a Banach algebra, then $A$ has no nontrivial idempotents if and only if for some (any) dense subset $F\subseteq A$, we have $sp(a)$ is connected for any $a\in F$, where $sp(a)$ is the spectrum of $a$ in $A$.
\end{lemma}
\begin{proof}
    Firstly, if $A$ has no nontrivial idempotents, we want to show that the spectrum of any element in $A$ is connected. Assume it is not true, namely, there exists an element $a\in A$, such that $sp(a)$ is disconnected, then there exist two disjoint open subsets $U,V\subseteq\mathbb{C}$ such that $sp(a)\subseteq U\cup V$, $sp(a)\cap U\neq \emptyset$ and $sp(a)\cap V\neq \emptyset$. Let $f$ be a function on $U\cup V$ such that $f|_U=0$ and $f|_V=1$, then $f$ is a holomorphic function on the neighborhood of $sp(a)$ and $f^2=f$. Taking holomorphic functional calculus, we obtain an idempotent $f(a)$ in $A$ and by spectral mapping theorem, we have $sp(f(a))=\{0,1\}$, which implies $f(a)$ is a nontrivial idempotent. Thus, we get a contradiction.\par
    For the other direction, assume $A$ has a nontrivial idempotent $a$, then $1-a$ is also a nontrivial idempotent and $sp(a)=sp(1-a)=\{0,1\}$. Since that $F$ is dense in $A$, there exists an element $b\in F$ such that $\|b-a\|<\min\{1/(4(2\|a\|+1)), \|a\| \}$ which implies $\|b^2-b\|<1/4$, thus $sp(b)\subset \{x\in \mathbb{C}: Re(x)\neq 1/2\}$. Let $\chi$ be a function such that $\chi(x)=1$ for $Re(x)>1/2$ and $\chi(x)=0$ for $Re(x)<1/2$. Since the holomorphic functional calculus by $\chi$ is norm continuous in the neighborhood of $a$, then there exists $\delta<\min\{1/(4(2\|a\|+1)), \|a\| \}$, such that 
  \begin{equation}\label{holocalcon}
     \|\chi(b')-a\|=\|\chi(b')-\chi(a)\|<1,
  \end{equation}
for some $b'\in F$ with $\|b'-a\|<\delta$. Thus $sp(b')\cap \{x\in \mathbb{C}: Re(x)<1/2\} \neq \emptyset$, otherwise, we have $\chi(b')=1$ which implies that $a$ is invertible. By the similar discussion for $1-a$, we obtain that $sp(b')\cap \{x\in \mathbb{C}: Re(x)>1/2\} \neq \emptyset$. In conclusion, $sp(b')$ is disconnected which contradict with the assumption that the spectrum of any element in $F$ is connected. 
\end{proof}

   For a discrete group $\Gamma$, a \textit{length function} on $\Gamma$ is a function $l:\Gamma \rightarrow [0,\infty)$ such that 
  \begin{enumerate}[(1)]
   \item $l(\gamma)=0$ if and only if $\gamma$ is the identity element,
   \item $l(\gamma^{-1})=l(\gamma)$ for any $\gamma\in \Gamma$,
   \item $l(\gamma_1+\gamma_2)\leq l(\gamma_1)+l(\gamma_2)$ for any $\gamma_1, \gamma_2 \in \Gamma$.
  \end{enumerate}
Let $e$ be the identity element of $\Gamma$, for any $n\geq 0$, denoted by $B_n(e)$ the set of all elements $\gamma$ in $\Gamma$ with $l(\gamma)\leq n$. 
    
\begin{definition}\label{defRD}
    Let $q\in[1,\infty]$ and $\Gamma$ be a discrete group. Say that $\Gamma$ has \textit{property $(RD)_q$} (with respect to a length function $l$), if there exists a polynomial $\mathcal{P}$ such that for any function $f\in \mathbb{C}\Gamma$ with support in $B_n(e)$, we have
\begin{equation*}
\|f\|_{\mathcal{B}(\ell^q(\Gamma))}\leq \mathcal{P}(n)\|f\|_{\ell^q(\Gamma)}.
\end{equation*}
\end{definition}

\begin{remark}
    The property $(RD)_q$ (more generally, defined for locally compact groups) was introduced by Liao and Yu (cf. \cite[Section 4]{LY1}) in order to compute the $K$-theory of $B^{q}_r(\Gamma)$ and $B^{q,\ast}_r(\Gamma)$. It is obvious that every group has property $(RD)_1$. When $q\in (2,\infty]$, Liao and Yu proved that a countable discrete group has property $(RD)_q$ with respect to a length function $l$ if and only if it has polynomial growth in $l$ (cf. \cite[Section 4]{LY1}). When $p=2$, the property $(RD)_2$ is called to be \textit{property $(RD)$} which was introduced by Jolissaint (cf. \cite{Jolissaint1}) and has important applications to Novikov conjecture (cf. \cite{ConnesMoscovici}) and Baum-Connes conjecture (cf. \cite{Lafforgue2002}). 
\end{remark}
    
    The following theorem is due to Lafforgue. Please refer to \cite[Theorem 4.4]{LY1} for two different proofs given by Liao, Yu and Pisier, respectively.

\begin{theorem}(V. Lafforgue)\label{RDqinterpolation}
    If $\Gamma$ is a discrete group with property $(RD)_q$ for some $q>1$, then it has $(RD)q^{\prime}$ for any $q^{\prime}\in(1,q)$. In particular, property $(RD)$ implies property $(RD)_q$ for any $q\in(1,2)$.
\end{theorem}



    
    Let $B$ be a unital Banach algebra, $A$ be a subalgebra of $B$, containing the unit element of $B$, we say that $A$ is \textit{stable under the holomorphic functional calculus} in $B$, if for every $a\in A$ and $f$ holomorphic in a neighborhood of $sp_{B}(a)$, the element $f(a)$ of $B$ lies in $A$, where $sp_{B}(a)$ is the spectrum of $a$ in $B$. We say that $A$ is a \textit{spectral invariant subalgebra} of $B$, if $sp_{A}(a)=sp_{B}(a)$ for any element $a\in A$. \par
    
    In \cite{Schweitzer1}, Schweitzer showed that the above two notions are equivalent.   
\begin{lemma}(\cite[Lemma 1.2]{Schweitzer1})\label{Schweitzerlemma}
    Let $B$ be a unital Banach algebra, $A$ be a Fr\'echet subalgebra of $B$, containing the unit element of $B$, then $A$ is stable under the holomorphic functional calculus in $B$ if and only if $A$ is spectral invariant in $B$.
\end{lemma}

    The significance of property $(RD)_q$ for groups is the following proposition proved by Liao and Yu.
\begin{proposition}(\cite[Proposition 4.6]{LY1})\label{1}
Let $p\in[1,\infty]$ and $q$ its dual number. Let $\Gamma$ be a discrete group with property $(RD)_q$ with respect to a length function $l$. Then for sufficiently large $t>0$, the space $S^t_q(\Gamma)$ of elements $f\in \ell^q(\Gamma)$ such that
\begin{equation*}
\|f\|_{S^t_q}\ :=\ \|\gamma \rightarrow (1+l(\gamma))^t f(\gamma)\|_{\ell^q(\Gamma)}<\infty
\end{equation*}
is a Banach algebra for the norm $\|\cdot\|_{S^t_q}$. It is contained in $B^{p,*}_r(\Gamma)$, $B^p_r(\Gamma)$ and $B^q_r(\Gamma)$, and stable under holomorphic functional calculus in each of these three algebras.
\end{proposition}

    Combining the above proposition with the Lemma \ref{Schweitzerlemma}, we have the following corollary.
\begin{corollary}\label{samespectrum}
    Let $p,q$ be as above, $\Gamma$ be a discrete group with property $(RD)_q$, then for any $f\in \mathbb{C}\Gamma$, we have
        $$sp_{B^{p,\ast}_r(\Gamma)}(f)=sp_{B^{p}_r(\Gamma)}(f)=sp_{B^{q}_r(\Gamma)}(f)=sp_{S^t_q(\Gamma)}(f)$$
\end{corollary}

    Now, we are ready to state and prove our main theorem.
\begin{theorem}\label{maintheorem}
    Let $p\in [1,\infty]$, $q$ be its dual number and $\Gamma$ be a discrete group. Assume $\Gamma$ has property $(RD)_q$, if $C^{\ast}_r(\Gamma)$ has no nontrivial idempotents, then both $B^p_r(\Gamma)$ and $B^q_r(\Gamma)$ also have no nontrivial idempotents.
\end{theorem}
\begin{proof}
    By the Corollary \ref{Cimplyp}, we have that $B^{p,\ast}_r(\Gamma)$ has no nontrivial idempotents, thus by the Lemma \ref{connectedspectrum}, the spectrum $sp_{B^{p,\ast}_r(\Gamma)}(f)$ is connected for any $f\in \mathbb{C}\Gamma$. If $\Gamma$ has property $(RD)_q$, then by the Corollary \ref{samespectrum}, we have $sp_{B^{p}_r(\Gamma)}(f)=sp_{B^{q}_r(\Gamma)}(f)=sp_{B^{p,\ast}_r(\Gamma)}(f)$ which are all connected for any $f\in \mathbb{C}\Gamma$. Thus, by the Lemma \ref{connectedspectrum} again, we obtain that both $B^p_r(\Gamma)$ and $B^q_r(\Gamma)$ have no nontrivial idempotents.
\end{proof}

    Combining the above theorem with the Theorem \ref{RDqinterpolation} and the fact that every group has property $(RD)_1$, we have the following corollary.
\begin{corollary}\label{RDidempotent}
    Let $\Gamma$ be a discrete group. Assume $\Gamma$ has property $(RD)$, if $C^{\ast}_r(\Gamma)$ has no nontrivial idempotents, then for any $p\in [1,\infty]$, $B^p_r(\Gamma)$ also has no nontrivial idempotents. 
\end{corollary}

\begin{example}\label{answertoPhillips} 
    We apply the main theorem to free groups $F_n$ for any positive integer $n$. The free group $F_n$ has property $(RD)$ was proved by Haagerup (cf. \cite{Haagerup1978}). And the reduced group $C^{\ast}$-algebra $C^{\ast}_r(F_n)$ has no nontrivial idempotents was proved by Pimsner and Voiculescu (cf. \cite{PV1}). Thus by the Corollary \ref{RDidempotent}, $B^p_r(F_n)$ has no nontrivial idempotents for any $p\in [1,\infty]$. This answer the Question \ref{Phillipsquestion} raised by Phillips (cf. \cite[part of problem 9.3]{Phillips1}). 
\end{example}

\begin{example}\label{hyperbolicgroupcase}
    In this example, we consider torsion free hyperbolic groups $\Gamma$. Hyperbolic groups have property $(RD)$ was proved by Jolissaint (cf. \cite{Jolissaint1}) and de la Harpe (cf. \cite{Harpe1988}). There are at least two different ways to prove that $C^{\ast}_r(\Gamma)$ has no nontrivial idempotents, one way is as a corollary of the Baum-Connes conjecture for hyperbolic groups which proved by Lafforgue (cf. \cite{Lafforgue2002}) and Mineyev, Yu (cf. \cite{MineyevYu}), another way is due to Puschnigg by using of local cyclic homology (cf. \cite{Puschnigg2002}). Thus by the Corollary \ref{RDidempotent}, $B^p_r(\Gamma)$ has no nontrivial idempotents for any $p\in [1,\infty]$.   
\end{example}

\begin{example}\label{polySL}
    For a torsion free discrete group $\Gamma$, if $\Gamma$ satisfies the Baum-Connes conjecture (actually, surjectivity of the assembly map is sufficient), then the reduced group $C^{\ast}$-algebra $C^{\ast}_r(\Gamma)$ has no nontrivial idempotents (cf. \cite{BCH1}). Thus by the Corollary \ref{RDidempotent}, for any $p\in [1,\infty]$, $B^p_r(\Gamma)$ has no nontrivial idempotents for every torsion free discrete group $\Gamma$ which has property $(RD)$ and satisfies the Baum-Connes conjecture. Apart from hyperbolic groups, such groups $\Gamma$ can also be finitely generated, torsion free groups with polynomial growth (cf. \cite{Jolissaint1}\cite{HK1}) and torsion free cocompact lattices of $SL(3,\mathbb{R})$ (cf. \cite{LafforgueRD}\cite{Lafforgue2002}).  
\end{example}

\bibliographystyle{plain}
\bibliography{idempotentsref}

\end{document}